\documentclass[10pt]{amsart}
\usepackage{ae,amsfonts,euscript,enumerate}
\usepackage{amsmath,euscript}
\usepackage{amssymb}
\usepackage{amsthm}
\usepackage{enumerate}
\usepackage{amsfonts}
\usepackage{comment}
\usepackage{colortbl}
\usepackage{a4wide}
\usepackage{amssymb,amsmath,array}

\newtheorem{thm}{Theorem}[section]
\newtheorem{lem}[thm]{Lemma}

\newtheorem{prop}[thm]{Proposition}

\newtheorem{defn}[thm]{Definition}

\def\co{{\mathcal O}}

\def\oqmm13{\co_q(M_{1,3})}
\def\oqm23{\co_q(M_{2,3})}

\usepackage{amssymb}
\usepackage{amsmath}
\usepackage[latin1]{inputenc}

\usepackage{fullpage}
\usepackage{amsmath,euscript}
\usepackage{amssymb}
\usepackage{amsthm}
\usepackage{enumerate}
\usepackage{amsfonts}
\usepackage{comment}
\usepackage{colortbl}
\usepackage{a4wide}
\usepackage{amssymb,amsmath,array}

\title[]{Shirshov's theorem and division rings\\
that are left algebraic over a subfield}
\author{Jason P.~Bell}
\thanks{The first-named author was supported by NSERC grant 31-611456.}

\address{
Department of Mathematics\\
Simon Fraser University\\
Burnaby, BC V5A 1S6\\
Canada}

\email{jpb@math.sfu.ca}
\email{ysharifi@sfu.ca}

\author{Vesselin Drensky}
\address{Institute of Mathematics and Informatics\\
Bulgarian Academy of Sciences\\
Sofia, Bulgaria}
\email{drensky@math.bas.bg}
\subjclass[2000]{}

\author{Yaghoub Sharifi}

\keywords{division rings, algebraic division rings, left algebraicity, Kurosch's problem, Levitzki's conjecture,
Shirshov's theorem, algebras with polynomial identity}

\subjclass[2010]{16K40, 16K20, 16R20, 05A05, 05E15}

\begin{document}
\bibliographystyle{plain}

\begin{abstract}  Let $D$ be a division ring.  We say that $D$ is \emph{left algebraic} over a (not necessarily central)
subfield $K$ of $D$ if every $x\in D$ satisfies a polynomial equation $x^n + \alpha_{n-1}x^{n-1}+\cdots +\alpha_0=0$
with $\alpha_0,\ldots ,\alpha_{n-1}\in K$.  We show that if $D$ is a division ring that is left algebraic over
a subfield $K$ of bounded degree $d$ then $D$ is at most $d^2$-dimensional over its center.
This generalizes a result of Kaplansky. For the proof we give a new version of the combinatorial theorem of Shirshov
that sufficiently long words over a finite alphabet contain either a $q$-decomposable subword or a high power of a non-trivial subword.
We show that if the word does not contain high powers then the factors in the $q$-decomposition may be chosen
to be of almost the same length.
We conclude by giving a list of problems for algebras that are left algebraic over a commutative subring.
\end{abstract}
\maketitle

\section{Introduction}

Kurosch \cite{Ku}, see also \cite[Problem 6.2.6]{R2} asked whether or not an algebra that is both finitely generated and algebraic over a field $k$
is necessarily finite-dimensional over $k$.  Kurosch's problem is a ring-theoretic analogue of Burnside's problem for groups
and was shown to have a negative answer by Golod and Shafarevich \cite{GS}.  In fact, Golod \cite{G} used their construction to give
an example of a finitely generated infinite group $G$ with the property that every element in $G$ has finite order,
giving a negative answer to  Burnside's problem.  As Rowen \cite[p. 116]{R2} points out, there are two special cases
of Kurosch's problem: the case that the algebra we consider is a division ring and the case that it is a nil ring.

Many examples of bizarre finitely generated algebraic algebras that are not finite-dimensional over their base fields now exist
\cite{BS, Sm1, Sm2, Sm3, LS}.  The strangest of these examples are due to Smoktunowicz, who showed a simple nil algebra (without $1$)
exists \cite{Sm2} and that there is a nil algebra (without $1$) with the property that the polynomial ring over it contains
a free algebra on two generators \cite{Sm3}.
Lenagan and Smoktunowicz \cite{LS} also showed that there is a finitely generated nil algebra (without $1$) that is infinite-dimensional
over its base field but has finite Gelfand-Kirillov dimension \cite{LS}.  Despite the large number of pathological examples of nil rings,
there are no similar pathological examples of algebraic division rings.
At the moment, all known examples of algebraic division rings have the property that every finitely generated subalgebra
is finite-dimensional over its center.

Kaplansky \cite{K} considered algebraic algebras that have the stronger property that there is a natural number $d$
such that every element of the algebra is algebraic of degree at most $d$.  With this stronger property,
one avoids the pathologies that arise when one considers algebras that are algebraic.

\begin{thm} {\rm (Kaplansky and Jacobson\footnote{In a footnote, Kaplansky thanks Jacobson for a key part of the proof.} \cite[Theorem 4]{K})}
A primitive algebraic algebra of bounded degree is finite-dimensional over its center.
\end{thm}

In fact, a primitive algebra that is finite-dimensional over its center is a matrix ring over a division ring by the Artin-Wedderburn theorem.
We consider an analogue of Kaplansky's result for division rings that are \emph{left algebraic} over some subfield.

\begin{defn}
{\em Let $A$ be a ring and let $B$ be a subring of $A$ such that $A$ is a free left $B$-module.
We say that $A$ is \emph{left algebraic} over $B$ if for every $a\in A$ there is some natural number $n$
and some elements $\alpha_0,\ldots ,\alpha_n\in B$ such that $\alpha_n$ is regular and
\[
\sum_{j=0}^n \alpha_j a^j \ = \ 0.
\]
}
\end{defn}

The left algebraic property has been used by the first-named author and Rogalski in investigating the existence
of free subalgebras of division rings \cite{Bell}.
We are able to give an analogue of Kaplansky's result in which we replace the algebraic property with being left algebraic over a subfield.

\begin{thm}
Let $d$ be a natural number, let $D$ be a division ring with center $Z(D)$, and let $K$ be a (not necessarily central) subfield of $D$.
If $D$ is left algebraic of bounded degree $d$ over $K$ then $[D:Z(D)]\le d^2$. \label{thm: main}
\end{thm}

We note that the bound of $d^2$ in the conclusion of the statement of Theorem \ref{thm: main} is the best possible.
For example, if one takes $D$ to be a cyclic division algebra that is $d^2$-dimensional over its center and we take $K$ to be the
center of $D$ then every element of $D$ is algebraic of degree at most $d$ over $K$.  The fact that $K$ is not necessarily central
complicates matters and as a result our proof is considerably different from Kaplansky's proof.
We rely heavily on combinatorial results on semigroups due to Shirshov \cite{Sh}.
Usually Shirshov's theorem is applied to finitely generated PI-algebras $R$. It gives that the sufficiently long words on
the generators contain either a $q$-decomposable subword or a high power of a non-trivial subword.
The existence of a multilinear polynomial identity
replaces the $q$-decomposable subword with a linear combination of words which are lower in the degree lexicographic order
and the algebra $R$ is spanned by words which behave like monomials in a finite number of commuting variables.
Here we go in the opposite direction and use the $q$-decomposable words to deduce that the algebra is PI.
For our purposes we establish a new version of Shirshov's theorem which states that the factors in the $q$-decomposition may be chosen
to be of almost the same length.
Using these combinatorial results, we are able to prove that every finitely generated subalgebra of $D$ satisfies a polynomial identity.
Then we use classical results of structure theory of PI-algebras to complete the proof of our main result.

\section{Combinatorics on words}

In this section, we recall some of the basic facts from combinatorics on words and use them to give a strengthening of Shirshov's theorem.

Let $M$ be the free monoid consisting of all words over a finite alphabet  $\{x_1,\ldots ,x_m\}$.  We put a degree lexicographic order
on all words in $M$ by declaring that
\[
x_1\succ x_2\succ \cdots \succ x_m.
\]
Given a word $w\in M$ and a natural number $q$, we say that $w$ is $q$-\emph{decomposable} if there exist $w_1,\ldots ,w_q\in M$ such that
$w=w_1w_2\cdots w_q$ and for all permutations $\sigma\in S_q$ with $\sigma\neq {\rm id}$ we have
\[
w_1w_2\cdots w_q\succ w_{\sigma(1)}w_{\sigma(2)}\cdots w_{\sigma(q)}.
\]
If in addition, we can choose $w_1,\ldots ,w_q$ such that $(q-1){\rm length}(w_i)<{\rm length}(w)$
for all $i\in \{1,\ldots , q\}$, we say that $w$ is \emph{strongly} $q$-\emph{decomposable}.
Shirshov  proved the following famous combinatorial theorem.

\begin{thm}
{\rm (Shirshov, \cite{Sh}, see also \cite[Lemma 4.2.5]{R})} Let $m$, $p$, and $q$ be natural numbers and
let $M$ be the free monoid generated by $m$ elements $x_1,\ldots ,x_m$.
Then there exists a positive integer $N(m,p,q)$, depending on $m$, $p$, and $q$, such that every word on $x_1,\ldots ,x_m$
of length greater than $N(m,p,q)$ has either a $q$-decomposable subword or has a non-trivial subword of the form $w^p$.
\label{Shirshov}
\end{thm}

By following the proof of Pirillo \cite{P}, we are able to give a strengthened version of Shirshov's theorem.
We first give some of the basic background from combinatorics on words.

Let $\Sigma=\{x_1,...,x_d\}$ be a finite alphabet.  We say that $w$ is a \emph{right infinite} word over
the alphabet $\Sigma$ if there is some map $f:\mathbb{N}\to \Sigma$ such that
\[
w=f(1)f(2)f(3)\cdots .
\]
We say that $v$ is a subword of the right infinite word $w$ if there exist natural numbers $i$ and $j$ with $i<j$ such that
\[
v=f(i)f(i+1)\cdots f(j).
\]
We say that the right infinite word $w$ is \emph{uniformly recurrent} if for each subword $v$ of $w$ there exists some natural number $N=N(v)$
such that for each $i\ge 1$ we have
$f(i)f(i+1)\cdots f(i+N)$ contains $v$ as a subword.  Given $i<j$, we let $v[i,j]$ denote the subword of $v$ that starts
at position $i$ and ends at position $j$.

We recall two classical results in the theory of combinatorics of words.
The first one is a consequence of K\"onig's infinity lemma in graph theory \cite{Ko} which
gives a sufficient condition for an infinite graph to have an infinitely long path, see e.g. \cite[p. 28, Exercise 41]{AS}.

\begin{thm} {\rm (K\"onig)} Let $\Sigma$ be a finite alphabet and let $S$ be an infinite subset of the free monoid $\Sigma^*$
generated by $\Sigma$.
Then there is a right infinite word $w$ over $\Sigma$ such that every subword of $w$
is a subword of some word in $S$.
\end{thm}

\begin{thm} {\rm (Furstenberg \cite{F}, see also \cite[p. 337, Exercise 22]{AS})}
Let $\Sigma$ be a finite alphabet and let $w$ be a right infinite word over $\Sigma$.
Then there is a right infinite uniformly recurrent word $u$ over $\Sigma$ such that every subword of $u$ is also a subword of $w$.
\end{thm}

Using these results, we are able to prove the following result.

\begin{thm}  Let $m$, $p$, and $q$ be natural numbers and let $M$ be a free monoid generated by $m$ elements $x_1,\ldots ,x_m$.
Then there exists a positive integer $N(m,p,q)$, depending on $m$, $p$, and $q$, such that every word on $x_1,\ldots ,x_m$
of length greater than $N(m,p,q)$ has either a strongly $q$-decomposable subword or has a non-trivial subword of the form $t^p$.
\label{Shirshov2}
\end{thm}

\begin{proof}
Suppose to the contrary that there are arbitrarily long words in $M$ that do not have a subword of the form $t^p$
or a strongly $q$-decomposable subword. Clearly, $q\geq 2$. Then by K\"onig's theorem there is a right infinite word $w$ over $\{x_1,\ldots ,x_m\}$
such that each finite subword $v$ of $w$ has the property that it does not have a subword of the form $t^p$
or a strongly $q$-decomposable subword.  By Furstenberg's theorem, there is a right infinite uniformly recurrent word $u$
such that each subword of $u$ has the property that it does not have a subword of the form $t^p$ or a strongly $q$-decomposable subword.

Let $\omega(n)$ denote the number of subwords of $u$ of length $n$.  Then $\omega(n)$ is not ${\rm O}(1)$,
since otherwise we would have $u$ is eventually periodic and thus it would have a subword of the form $t^p$.
Hence there is some natural number $N$ such that there are at least $q$ distinct subwords of $u$ of length $N$.
Let $w_1\succ w_2\succ \cdots \succ w_q$ be $q$ such words of length $N$.

Since $w_1,\ldots ,w_q$ are uniformly recurrent in $u$, there is some natural number $L$ such that
$w_1,\ldots ,w_q$ occur in the interval $u[i,i+L]$ for each $i$.  Then there is an occurrence of
$w_1$ somewhere in $u[1,1+L]$.  We let $j_1\in \{1,\ldots ,1+L\}$ denote the position of the first
letter of $w_1$ in some occurrence in $u[1,1+L]$.  Then there is an occurrence of $w_2$
somewhere in $u[2Lq+1,2Lq+L+1]$; we let $j_2$ denote its starting position.  Continuing in this
manner, we define natural numbers $j_1,\ldots ,j_q$ such that $j_i\in [2Lq(i-1)+1,2Lq(i-1)+L+1]$ for $1\le i \le q$
and such that $w_i=u[j_i,j_i+N-1]$.

We define $u_i:=u[j_i,j_{i+1}-1]$ for $i\in \{1,\ldots ,q-1\}$ and we define $u_q:=u[j_q,j_q+2Lq]$.
Then by construction,
${\rm length}(u_i) < L(2q+1)$ for all $i$ and $w_i$ is the initial subword of length $N$ of $u_i$ for all $i$.
In particular, $u_1\cdots u_q \succ u_{\sigma(1)}\cdots u_{\sigma(q)}$ for all $\sigma\neq {\rm id}$.
Finally, note that $j_1\leq L$, $j_q\geq 2Lq(q-1)+1$ and hence
\[
{\rm length}(u_1\cdots u_q)=2Lq+j_q-j_1+1\ge L(2q^2-1)+1>(q-1)L(2q+1)> (q-1){\rm length}(u_i)
\]
for $i\in \{1,\ldots ,q\}$, which contradicts the assumption that $u$ does not contain strongly $q$-decomposable subwords.
\end{proof}

\section{Proof of Theorem \ref{thm: main}}

In this section we prove Theorem \ref{thm: main}. Let $D$ be a division ring with center $k$. The proof is done by a series of reductions.
We first prove that if $D$ is left algebraic of bounded degree over a subfield $K$, then every finitely generated $k$-subalgebra
satisfies a standard polynomial identity.  We then use a theorem of Albert \cite{Albert} to prove that $D$ must satisfy a standard identity.
From there, we prove the main theorem by embedding $D$ in a matrix ring and looking at degrees of minimal polynomials.

We now prove the first step in our reduction.

\begin{lem} Let the division algebra $D$ be left algebraic of bounded degree $d$ over a (not necessarily central) subfield $K$.
If $m$ is a natural number, then there is a positive integer $C=C(m,d)$, depending only on $d$ and $m$,
such that every $k$-subalgebra of $D$ that is generated by $m$ elements satisfies the standard polynomial identity of degree $C$
\[
s_C(y_1,\ldots,y_C)=\sum_{\pi\in S_C}\text{\rm sign}(\pi)y_{\pi(1)}\cdots y_{\pi(C)}=0.
\]
\label{locPI}
\end{lem}

\begin{proof} Let $x_1,\ldots ,x_m$ be $m$ elements of $D$.
Consider the $k$-subalgebra $A$ of $D$ generated by $x_1,\ldots ,x_m$.   We put a degree lexicographic order on all words over
$\{x_1,\ldots ,x_m\}$ by declaring that
\[
x_1\succ x_2\succ \cdots \succ x_m.
\]
Let $N=N(m,d,d)$ be a positive integer satisfying the conclusion of the statement of Theorem \ref{Shirshov} in which we take $p=q=d$.
We claim that the left $K$-vector space $V:=KA$ is spanned by all words in $x_1,\ldots ,x_m$ of length at most $N$.

To see this, suppose that the claim is false and let $w$ be the smallest degree lexicographically word with the property that
it is not in the left $K$-span of all words of length at most $N$.

Then $w$ must have length strictly greater than $N$ and so by Theorem \ref{Shirshov2}, either $w$ has a strongly $d$-decomposable
subword or $w$ has a non-trivial subword of the form $u^d$.

If $w$ has a non-trivial subword of the form $u^d$ then we can write $w=w_1 u^d w_2$.
Notice that conjugation by $w_1$ gives an automorphism of $D$ and so $D$ must also be left algebraic of bounded degree $d$
over the subfield $F:=w_1^{-1}Kw_1$.  Notice that the sum
\[
Fu^d+Fu^{d-1}+\cdots +F
\]
is not direct and thus we can find $\alpha_{0},\ldots ,\alpha_{d-1}\in K$ such that
\[
u^d = w_1^{-1}\alpha_{d-1}w_1 u^{d-1}+\cdots + w_1^{-1} \alpha_0 w_1.
\]
Thus
\begin{eqnarray*} w &=&  w_1 u^d w_2 \\
&=&  w_1\left( w_1^{-1}\alpha_{d-1}w_1 u^{d-1}+\cdots + w_1^{-1} \alpha_0 w_1\right) w_2 \\
&=& \alpha_{d-1}w_1 u^{d-1}w_2+\cdots +\alpha_0 w_1w_2 \\
&\in & \sum_{v\prec w} Kv. \end{eqnarray*}
By the minimality of $w$, we get an immediate contradiction.

Similarly, if $w$ has a strongly $d$-decomposable subword, then we can write
\[
w=w_1u_1\cdots u_d w_2
\]
where we have
\[
u_1\cdots u_d \succ u_{\sigma(1)}\cdots u_{\sigma(d)}\]
for all non-identity permutations $\sigma$ in $S_d$ and such that
$(d-1){\rm length}(u_i)<{\rm length}(u_1\cdots u_d)$ for each $i$.
As before, we let $F=w_1^{-1}Kw_1$.  Given a subset $S\subseteq \{1,\ldots ,d\}$, we let $u_S=\sum_{j\in S} u_j$.
Then for each subset $S$ of $\{1,\ldots ,d\}$, we can find $\alpha_{0,S},\ldots ,\alpha_{d-1,S}\in K$ such that
\[
u_S^d = w_1^{-1}\alpha_{d-1}w_1 u_S^{d-1}+\cdots + w_1^{-1} \alpha_0 w_1.
\]
The condition $(d-1){\rm length}(u_i)<{\rm length}(u_1\cdots u_d)$ implies that if $k<d$, then
\[
{\rm length}(u_{i_1}\cdots u_{i_k})<{\rm length}(u_1\cdots u_d)
\]
and hence $u_{i_1}\cdots u_{i_k}\prec u_1\cdots u_d$
for all summands of $u_S^k$, $k<d$.
Notice that
\[
\sum_{S\subseteq \{1,\ldots ,j\}}  (-1)^{d-|S|} u_S^d = u_1\cdots u_d
+ \sum_{\stackrel{\sigma\in S_d}{\sigma\neq {\rm id}}} u_{\sigma(1)}\cdots u_{\sigma(d)},
\]
and so
\begin{eqnarray*}
w &=&  w_1u_1\cdots u_d w_2 \\
&=& -\sum_{\stackrel{\sigma\in S_d}{\sigma\neq {\rm id}}} w_1 u_{\sigma(1)}\cdots u_{\sigma(d)} w_2
+ \sum_{S\subseteq \{1,\ldots ,d\}} \sum_{j=0}^{d-1} (-1)^{d-|S|}\alpha_{j,S}w_1 u_S^{j}w_2 \\
&\in & \sum_{v\prec w} Kv.
\end{eqnarray*}
By the minimality of $w$, we get a contradiction.

Thus $V=KA$ is indeed spanned by all words over $\{x_1,\ldots ,x_m\}$ of length at most $N$.
Consequently, $V$ is at most $(1+m+m^2+\cdots+m^N)$-dimensional as a left $K$-vector space.
The right multiplication $r_a$ by $a\in A$ of the elements of $V$ commutes with the left multiplication by
elements of $K$. Hence $r_a$ acts as a linear operator on the left $K$-vector space $V$ and
$A$ embeds in the opposite algebra ${\rm End}_K(V)^{\rm op}$ of ${\rm End}_K(V)$.
In this way $A$ embeds in the ring of $n\times n$ matrices over $K$ for some $n\le 1+m+m^2+\cdots+m^N$.
Thus taking $C=2(1+m+m^2+\cdots+m^N)$ and invoking the Amitsur-Levitzski theorem \cite[Theorem 1.4.1]{R}, we obtain the desired result.
\end{proof}

\begin{lem}
Let $D$ be a division algebra which is left algebraic of bounded degree over a subfield $K$.
Then every finitely generated division $k$-subalgebra $E$ of $D$ is finite-dimensional over the center of $E$.
\label{fdim division subalgebras}
\end{lem}

\begin{proof}
Let $E$ be generated (as a division $k$-algebra) by
$x_1,\ldots,x_m$, and let $A$ be the $k$-subalgebra of $E$ generated by these elements, i.e., $A$ is the $k$-vector space
spanned by all words over $\{x_1,\ldots ,x_m\}$.
By Lemma \ref{locPI} the algebra $A$ satisfies a standard identity $s_C=0$ of degree $C=C(m,d)$.
Since $A$ is a prime PI-algebra, by the modern form of Posner's theorem, see e.g.
Theorem 6.5, p. 160 of Formanek's part of \cite{DF},
if $Z(A)$ is the center of $A$, then $Q(A)=(Z(A)\setminus \{0\})^{-1}A$ is a finite-dimensional central
simple $(Z(A)\setminus \{0\})^{-1}Z(A)$-algebra
which satisfies exactly the same polynomial identities as $A$.
Since $A$ is a subalgebra of $E$, the natural embedding $\iota:A\to E$ extends to an injection $\iota: Q(A)\to E$.
Since $\iota(Q(A))$ is a subring of the division ring $E$, it is a central simple algebra without zero-divisors, i.e. it is a division
algebra.
As a division $k$-algebra $\iota(Q(A))$ is generated by the same elements $x_1\ldots,x_m$ as the division $k$-algebra $E$.
Hence we obtain that $\iota(Q(A))=E$ and $E$ is isomorphic to $Q(A)$.
Since $Q(A)$ is finite-dimensional over its
center, the same holds for $E$.
\end{proof}

\begin{prop} Let $D$ be a division algebra that is left algebraic of bounded degree $d$ over a maximal subfield $K$.
Then $D$ satisfies the standard polynomial identity of degree $C=C(2,d)$, where
$C=C(2,d)$ is a constant satisfying the conclusion of the statement of Lemma \ref{locPI}.\label{prop: 1}
\end{prop}

\begin{proof}
Let $k$ be the center of $D$. The $T$-ideal
of the polynomial identities of the division $k$-algebra $D$ is the intersection
of the $T$-ideals of all finitely generated division $k$-subalgebras.
If $D$ does not satisfy the standard identity $s_C=0$,
then there exists a finitely generated division $k$-subalgebra $E$ of $D$
such that $E$ does not satisfy the identity $s_C=0$.
By Lemma \ref{fdim division subalgebras} $E$ is finite-dimensional over its center $Z(E)$.
By a result of Albert \cite{Albert}, $E$
is generated by two elements as a $Z(E)$-algebra. Let $a$ and $b$ be the generators of the $Z(E)$-algebra $E$.
By Lemma \ref{locPI} the $k$-algebra $A$ generated by $a$ and $b$ satisfies the standard identity of degree $C=C(2,d)$.
Since the center $k$ of $D$ is contained in the center $Z(E)$ of $E$ and $a,b\in E$,
we have that $Z(E)A\subseteq E$. Since $E$ is generated as a $Z(E)$-algebra by $a$ and $b$ we conclude that
$E=Z(E)A$.  Thus we have a surjective ring homomorphism
\[
Z(E)\otimes_kA \to E
\]
and since $A$ satisfies the standard identity of degree $C$, the same holds for $Z(E)\otimes_kA$ and $E$, a contradiction.
Thus $D$ satisfies the standard polynomial identity of degree $C$.
\end{proof}

We are now ready to prove our main result.  We have already shown that if a division ring $D$ is left algebraic of bounded degree
over a subfield $K$ then $D$ satisfies a polynomial identity and hence is finite-dimensional over its center.
The only thing that remains is to get the upper bound that is claimed in the statement of Theorem \ref{thm: main}.
This is not difficult if the subfield $K$ is separable over $k$ as one can use a theorem of Brauer and Albert \cite[15.16]{Lam}
which states that if $D$ is a division ring of dimension $n^2$ over its center $k$, then there exist elements $a,b\in D$ such that
$D$ has a $k$-basis $\{a^iba^j ~:~ i,j=0,1,\ldots,n-1\}$.
The inseparable case presents greater difficulty.

\begin{proof}[Proof of Theorem \ref{thm: main}] It is no loss of generality to assume that $K$ is a maximal subfield of $D$.
Let $k$ denote the center of $D$.  By Proposition \ref{prop: 1}, $D$ satisfies a polynomial identity
and hence it is finite-dimensional over $k$ \cite[Theorem 1.5.16]{R}.

Let $n=\sqrt{[D:k]}$.  Then $[D:K]=n$ \cite[Proposition 1, p. 180]{J} and we must show that $d\ge n$.
We note that $D$ has a separable maximal subfield $L=k(x)$ \cite[Theorem 1, p. 181]{J} and $D$ is a faithful simple
left $D\otimes_k L$-module, via the rule
\[
(\alpha \otimes  x^j)(\beta) \mapsto \alpha \beta x^j
\]
for $j\ge 0$ and $\alpha,\beta\in D$.
We let $T\in {\rm End}_K(D)$ be defined by $T(\alpha)=\alpha x$. If $c_0,\ldots ,c_{n-1}\in K$ then
\[
\left(c_0{\rm id}+\cdots +c_{n-1}T^{n-1}\right)(\alpha) = \left( \sum_{i=0}^n c_i\otimes x^i\right)(\alpha).
\]
Since $D$ is a faithful $D\otimes_k L$-module, we see that if
\[
c_0{\rm id}+\cdots +c_{n-1}T^{n-1}=0
\]
then $c_0=\ldots =c_{n-1}=0$ and so the operators
${\rm id}, T, \ldots ,T^{n-1}$ are (left) linearly independent over $K$.

We claim that there exists some $y\in D$ such that the sum
\[
K+KT(y)+\cdots +KT^{n-1}(y)
\]
is direct. To see this, we regard $D$ as a left $K[X]$-module, with action given by $f(X)\cdot \alpha \mapsto f(T)(\alpha)$
for $f(X)\in K[X]$ and $\alpha\in D$.  Let $g(X)$ denote the minimal polynomial of $T$ over $k$.  Then $g(X)$ annihilates $D$ and thus
$D$ is a finitely generated torsion $K[X]$-module. By the fundamental theorem for finitely generated modules
over a principal ideal domain, there exists some $y\in D$ such that
\begin{equation}
\label{eq: 11}\{ f(X)\in K[X]~:~f(X)\cdot y = 0 \} = \{f(X)\in K[X]~:~f(X)\cdot \alpha=0~{\rm for ~all~}\alpha\in D\}.
\end{equation}
If the sum $K+KT(y)+\cdots +KT^{n-1}(y)$ is not direct, then we can find a polynomial $f(X)\in K[X]$ of degree at most $n-1$
such that $f(T)\cdot y=0$.  Thus $f(T)\cdot \alpha=0$ for all $\alpha\in D$ by Equation (\ref{eq: 11}),
which contradicts the fact that  the operators
${\rm id}, T, \ldots ,T^{n-1}$ are (left) linearly independent over $K$.

Hence the sum
\[
K+KT(y)+\cdots +KT^{n-1}(y)=K+Kyx+\cdots + Kyx^{n-1}
\]
is direct.
Let $u=yxy^{-1}$.  Then $K+Ku+\cdots +Ku^{n-1}$ is direct.  But by assumption, every element of $D$ is left algebraic over $K$
of degree at most $d$ and thus $n\le d$.  The result follows.
\end{proof}

\section{Problems}

Unlike the algebraic property, which has been extensively studied in rings, the left algebraic property appears to be new.
Many of the important open problems for algebraic algebras have analogues in which the algebraic property is replaced by being left algebraic.
We pose a few problems.
\begin{enumerate}
\item Is it true that a division ring that is finitely generated over its center and left algebraic over
some subfield is finite-dimensional over its center?
\item Let $k$ be an algebraically closed field and let $A$ be a finitely generated noetherian $k$-algebra
that does not satisfy a polynomial identity.  Is it possible for the quotient division algebra of $A$ to be left algebraic over some subfield?
\item We note that the right algebraic property can be defined analogously.  If a division ring $D$ is left algebraic
over a subfield $K$ must $D$ also be right algebraic over $K$?  We believe that this question has probably been posed before,
but we are unaware of a reference.
\end{enumerate}

\section*{Acknowledgment}
We thank Lance Small, who suggested this problem to the first-named author.

\end{document}